\documentclass[twoside,11pt]{amsart}

\usepackage[cmtip,all]{xy}
\usepackage{xspace, enumerate, times, color}
\usepackage{amssymb, hyperref}

\usepackage{graphicx}
\input xypic
\input xy
\xyoption{all}
\def\sqr#1#2{{\vcenter{\hrule height.#2pt
        \hbox{\vrule width.#2pt height#1pt \kern#1pt
                \vrule width.#2pt}
        \hrule height.#2pt}}}

\numberwithin{equation}{section}

\newtheorem{theorem}{theorem}[section]
\newtheorem{lemma}[theorem]{Lemma}

\newtheorem{open Problem}[theorem]{Open Problem}

\newtheorem{remark}[theorem]{Remark}

\newcommand{\be}{\begin{equation*}}
\newcommand{\ee}{\end{equation*}}
\newcommand{\bee}{\begin{equation}}
\newcommand{\eee}{\end{equation}}

\definecolor{lighterorange}{cmyk}{0,0.42,0.66,0.0}

\title[On partitions of $\mathbb{Z}_{m}$ with the same representation function]{On partitions of $\mathbb{Z}_{m}$ with the same representation function}
\author{Cui-Fang SUN and Meng-Chi XIONG}

\begin{document}

\date{2020-7-1\\E-mail:  cuifangsun@163.com,\;\;  mengchixiong@126.com}

\maketitle

\begin{abstract}
For any positive integer $m$, let $\mathbb{Z}_{m}$ be the set of residue classes modulo $m$. For $A\subseteq \mathbb{Z}_{m}$ and $\overline{n}\in \mathbb{Z}_{m}$, let $R_{A}(\overline{n})$ denote the number of solutions of $\overline{n}=\overline{a}+\overline{a'}$ with unordered pairs $(\overline{a}, \overline{a'})\in A \times A$. In this paper, we prove that if $m=2^{\alpha}$ with $\alpha\neq 2$, $A\cup B=\mathbb{Z}_{m}$ and $|A\cap B|=2$, then $R_{A}(\overline{n})=R_{A}(\overline{n})$ for all $\overline{n}\in \mathbb{Z}_{m}$ if and only if $B=A+\overline{\frac{m}{2}}$.

\noindent{{\bf Keywords:}\hspace{2mm} Representation function, partition, residue class.}
\end{abstract}

\maketitle

\section{Introduction}
Let $\mathbb{N}$ be the set of all nonnegative integers. For $S\subseteq \mathbb{N}$ and $n\in S$, let the representation function $R'_{S}(n)$ denote the number of solutions of the equation $s+s'=n$ with $s\leq s'$ and $s, s'\in S$. S\'{a}rk\H{o}zy asked whether there exist two subsets $A, B$ of $\mathbb{N}$ with $|(A\cup B)\backslash (A\cap B)|=\infty$ such that $R'_{A}(n)=R'_{B}(n)$ for all sufficiently large integers $n$. In 2003, Chen and Wang \cite{CB} showed that the set of positive integers can be partitioned into two subsets $A$ and $B$ such that $R'_A(n)=R'_B(n)$ for all $n\geq 3$. There are many other related results (see \cite{D, KS, LT, Q3, TC, TL} and the references therein).

For a positive integer $m$, let $\mathbb{Z}_{m}$ be the set of residue classes modulo $m$. For any residue classes $\overline{a}, \overline{b}\in \mathbb{Z}_{m}$, there exist two integers $a', b'$ with $0\leq a', b'\leq m-1$ such that $\overline{a'}=\overline{a}$ and $\overline{b'}=\overline{b}$. We define the ordering $\overline{a}\leq \overline{b}$ if $a'\leq b'$. For $A\subseteq \mathbb{Z}_{m}$ and $\overline{n}\in \mathbb{Z}_{m}$, let $R_{A}(\overline{n})$ denote the number of solutions of $\overline{n}=\overline{a}+\overline{a'}$ with $\overline{a}\leq \overline{a'}$ and $\overline{a}, \overline{a'}\in A$. For $\overline{n}\in \mathbb{Z}_{m}$ and $A\subseteq \mathbb{Z}_{m}$, let $\overline{n}+A=\{\overline{n}+\overline{a}: \overline{a}\in A\}$. For $A, B\subseteq \mathbb{Z}_{m}$ and $\overline{n}\in \mathbb{Z}_{m}$, let $R_{A, B}(\overline{n})$ be the number of solutions of $\overline{n}=\overline{a}+\overline{b}$ with $\overline{a}\in A$ and $\overline{b}\in B$.
The characteristic function of $A\subseteq \mathbb{Z}_{m}$ is denoted by
$$\chi_{A}(n)=\begin{cases}1  &\overline{n}\in A,\\
0 & \overline{n}\not\in A.
\end{cases}$$

In 2012, Yang and Chen \cite{YC} determined all sets $A, B\subseteq \mathbb{Z}_{m}$ with $|(A\cup B)\backslash (A\cap B)|=m$ such that $R_{A}(\overline{n})=R_{B}(\overline{n})$ for all $\overline{n}\in \mathbb{Z}_{m}$. In 2014, Qu \cite{Q1, Q2} studied more general forms of these results. In 2014, Kiss et al. \cite{K} generalized some results to the finite Abelian group. In 2017, Yang and Tang \cite{YT} determined all sets $A, B\subseteq \mathbb{Z}_{m}$ with $|(A\cup B)\backslash (A\cap B)|=2$ or $m-1$ such that  $R_{A}(\overline{n})=R_{B}(\overline{n})$ for all $\overline{n}\in \mathbb{Z}_{m}$.

 In this paper, we consider the partitions of $\mathbb{Z}_{m}$ with $A\cup B=\mathbb{Z}_{m}$ and $|A\cap B|=2$ and obtain the following result:

\begin{theorem}\label{thm1}
Let $\alpha\neq 2$ be an integer and $m=2^{\alpha}$. Let $A, B\subseteq \mathbb{Z}_{m}$ with $A\cup B=\mathbb{Z}_{m}$ and $|A\cap B|=2$. Then $R_{A}(\overline{n})=R_{B}(\overline{n})$ for all $\overline{n}\in \mathbb{Z}_{m}$ if and only if $B=A+\overline{\frac{m}{2}}$.
\end{theorem}

\begin{remark}
Let $m=2^{2}$ and $\mathbb{Z}_{m}=\{\overline{0}, \overline{1}, \overline{2}, \overline{3}\}$. Let $A=\{\overline{0}, \overline{1}, \overline{2}\}$ and $B=\{\overline{0}, \overline{1}, \overline{3}\}$. Then $B\neq A+\overline{\frac{m}{2}}$ and $R_{A}(\overline{n})=R_{B}(\overline{n})$ for all $\overline{n}\in \mathbb{Z}_{m}$.
\end{remark}

\section{Lemmas}

\begin{lemma}\label{lem1}
Let $m$ be a positive even integer. Let $A, B\subseteq \mathbb{Z}_{m}$ with $A\cup B=\mathbb{Z}_{m}$ and $|A\cap B|=2$. If $R_{A}(\overline{n})=R_{B}(\overline{n})$ for all $\overline{n}\in \mathbb{Z}_{m}$, then $|A|=|B|=\frac{m}{2}+1$.
\end{lemma}

\begin{proof}
If $R_{A}(\overline{n})=R_{B}(\overline{n})$ for all $\overline{n}\in \mathbb{Z}_{m}$, then
$$\binom{|A|}{2}+|A|=\sum\limits_{\overline{n}\in \mathbb{Z}_{m}}R_{A}(\overline{n})=\sum\limits_{\overline{n}\in \mathbb{Z}_{m}}R_{B}(\overline{n})=\binom{|B|}{2}+|B|.$$
Thus $|A|=|B|$. Noting that
$$|A|+|B|=|A\cup B|+|A\cap B|=m+2,$$
we have $|A|=|B|=\frac{m}{2}+1$.

This completes the proof of Lemma \ref{lem1}.
\end{proof}

\begin{lemma}\label{lem2}
Let $m$ be a positive even integer. Let $A, B\subseteq \mathbb{Z}_{m}$ with $A\cup B=\mathbb{Z}_{m}$ and $A\cap B=\{\overline{r_{1}}, \overline{r_{2}}\}$. If $R_{A}(\overline{n})=R_{B}(\overline{n})$ for all $\overline{n}\in \mathbb{Z}_{m}$, then
$$\chi_{A}(n-r_{1})+\chi_{A}(n-r_{2})=1+R_{\{\overline{r_{1}}, \overline{r_{2}}\}}(\overline{n}), \text{ if } 2\nmid n$$
and
$$\chi_{A}(n-r_{1})+\chi_{A}(n-r_{2})=2+R_{\{\overline{r_{1}}, \overline{r_{2}}\}}(\overline{n})-\chi_{A}(\frac{n}{2})-\chi_{A}(\frac{n+m}{2}), \text{ if } 2\mid n.$$
\end{lemma}

\begin{proof} For any $\overline{n}\in \mathbb{Z}_{m}$, without loss of generality, we may suppose that $0\leq n\leq m-1$.
Noting that $B=(\mathbb{Z}_{m}\backslash A)\cup \{\overline{r_{1}}, \overline{r_{2}}\}$, we have
\begin{eqnarray*}
R_{B}(\overline{n})&=& R_{\mathbb{Z}_{m}\backslash A}(\overline{n})+R_{\mathbb{Z}_{m}\backslash A, \{\overline{r_{1}}, \overline{r_{2}}\}}(\overline{n})+
 R_{\{\overline{r_{1}}, \overline{r_{2}}\}}(\overline{n})\\
&=& |\{(a, a'): \overline{a}, \overline{a'}\in\mathbb{Z}_{m}\backslash A, 0\leq a\leq a'\leq m-1, a+a'=n \text{ or } a+a'=n+m\}|\\
&& \hspace{2mm}+\sum\limits_{i=1}^{2}(1-\chi_{A}(n-r_{i}))+R_{\{\overline{r_{1}}, \overline{r_{2}}\}}(\overline{n})\\
&=& \sum\limits_{0\leq i\leq \frac{n}{2}}(1-\chi_{A}(i))(1-\chi_{A}(n-i))+\sum\limits_{n+1\leq i\leq \frac{n+m}{2}}(1-\chi_{A}(i))(1-\chi_{A}(n-i))\\
&&\hspace{2mm} +\sum\limits_{i=1}^{2}(1-\chi_{A}(n-r_{i}))+R_{\{\overline{r_{1}}, \overline{r_{2}}\}}(\overline{n})\\
&=& \sum\limits_{0\leq i\leq \frac{n}{2}}1-\sum\limits_{0\leq i\leq n}\chi_{A}(i)-\chi_{A}(\frac{n}{2})+\sum\limits_{0\leq i\leq \frac{n}{2}}\chi_{A}(i)\chi_{A}(n-i)+\sum\limits_{n+1\leq i\leq \frac{n+m}{2}}1\\
&&\hspace{2mm} -\sum\limits_{n+1\leq i\leq m-1}\chi_{A}(i)-\chi_{A}(\frac{n+m}{2})+\sum\limits_{n+1\leq i\leq \frac{n+m}{2}}\chi_{A}(i)\chi_{A}(n-i)\\
&&\hspace{2mm} +\sum\limits_{i=1}^{2}(1-\chi_{A}(n-r_{i}))+R_{\{\overline{r_{1}}, \overline{r_{2}}\}}(\overline{n})
\end{eqnarray*}
\begin{eqnarray*}
&=& \sum\limits_{0\leq i\leq \frac{n}{2}}1+\sum\limits_{n+1\leq i\leq \frac{n+m}{2}}1-|A|-\chi_{A}(\frac{n}{2})-\chi_{A}(\frac{n+m}{2})+R_{A}(\overline{n})
\\
&&\hspace{2mm}+\sum\limits_{i=1}^{2}(1-\chi_{A}(n-r_{i}))+R_{\{\overline{r_{1}}, \overline{r_{2}}\}}(\overline{n}).\\
\end{eqnarray*}
Since $R_{A}(\overline{n})=R_{B}(\overline{n})$ for all $\overline{n}\in \mathbb{Z}_{m}$, we have
\begin{eqnarray*}0&=& \sum\limits_{0\leq i\leq \frac{n}{2}}1+\sum\limits_{n+1\leq i\leq \frac{n+m}{2}}1-|A|-\chi_{A}(\frac{n}{2})-\chi_{A}(\frac{n+m}{2})\\
&&\hspace{2mm}+\sum\limits_{i=1}^{2}(1-\chi_{A}(n-r_{i}))+R_{\{\overline{r_{1}}, \overline{r_{2}}\}}(\overline{n}).
\end{eqnarray*}
If $2\nmid n$, then by Lemma \ref{lem1} we have
$$\chi_{A}(n-r_{1})+\chi_{A}(n-r_{2})=1+R_{\{\overline{r_{1}}, \overline{r_{2}}\}}(\overline{n}).$$
If $2\mid n$, then  by Lemma \ref{lem1} we have
$$\chi_{A}(n-r_{1})+\chi_{A}(n-r_{2})=2+R_{\{\overline{r_{1}}, \overline{r_{2}}\}}(\overline{n})-\chi_{A}(\frac{n}{2})-\chi_{A}(\frac{n+m}{2}).$$

This completes the proof of Lemma \ref{lem2}.
\end{proof}

\begin{lemma}\label{lem3}
Let $m$ be a positive even integer. Let $A, B\subseteq \mathbb{Z}_{m}$ with $A\cup B=\mathbb{Z}_{m}$ and $A\cap B=\{\overline{r}, \overline{r+\frac{m}{2}}\}$. Then $R_{A}(\overline{n})=R_{B}(\overline{n})$ for all $\overline{n}\in \mathbb{Z}_{m}$ if and only if $B=A+\overline{\frac{m}{2}}$.
\end{lemma}

\begin{proof} If $B=A+\overline{\frac{m}{2}}$, then it is clear that $R_{A}(\overline{n})=R_{B}(\overline{n})$ for all $\overline{n}\in \mathbb{Z}_{m}$.
Now we suppose that $R_{A}(\overline{n})=R_{B}(\overline{n})$ for all $\overline{n}\in \mathbb{Z}_{m}$. It is sufficient to prove that for all integers $k$ with $\overline{k}\neq \overline{r}$ and $\overline{k}\neq \overline{r+\frac{m}{2}}$, we have
\begin{equation}\label{2.1}
\chi_{A}(k)+\chi_{A}(k+\frac{m}{2})=1.
\end{equation}
We will discuss the following two cases according to $m$.

{\bf Case 1.} $m\equiv 2\pmod 4$. Then $\frac{m}{2}$ is odd and $R_{\{\overline{r}, \overline{r+\frac{m}{2}}\}}(\overline{2k})=0$. By Lemma \ref{lem2}, we have
$$\chi_{A}(2k-r)+\chi_{A}(2k-(r+\frac{m}{2}))=2-\chi_{A}(k)-\chi_{A}(k+\frac{m}{2})$$
and
$$\chi_{A}((2k+\frac{m}{2})-r)+\chi_{A}((2k+\frac{m}{2})-(r+\frac{m}{2}))=1.$$
Thus
$$\chi_{A}(k)+\chi_{A}(k+\frac{m}{2})=1. $$

{\bf Case 2.} $m\equiv 0 \pmod 4$. Let $n, t$ be any integers with $n-t=r$. By Lemma \ref{lem2}, we have
\begin{eqnarray}\label{2.2}
 &&\chi_{A}(n-r)+\chi_{A}(n-(r+\frac{m}{2}))=1, \text { if }  2\nmid n  \nonumber \\
\Longleftrightarrow && \chi_{A}(t)+\chi_{A}(t+\frac{m}{2})=1, \text { if }  t\equiv r+1\pmod 2.
\end{eqnarray}

If $k\equiv r+1\pmod 2$, then we choose $t=k$ in (\ref{2.2}) and (\ref{2.1}) is proved.

Now we suppose that $k\equiv r\pmod 2$. For any integer $k_{1}$, let
$$ a_{i, k_{1}}=2^{i+1}k_{1}+2^{i}(r+1)-(2^{i}-1)r, \; i=1, 2, \ldots.$$
It is clear that \begin{equation}\label{2.3}
a_{i+1, k_{1}}=2a_{i, k_{1}}-r, \; i=1, 2, \ldots
\end{equation}
and
$$\{a_{i, k_{1}}: k_{1}\in \mathbb{Z}, i\in\mathbb{Z}^{+}\}\subseteq 2\mathbb{Z}+r.$$
On the other hand, for any $b\in 2\mathbb{Z}+r$, there exist integers $q, c$ with $q\geq 1$ and $2\nmid c$ such that $b=2^{q}c+r$. Thus
$$b-r+2^{q}r-2^{q}(r+1)=2^{q}c+2^{q}r-2^{q}(r+1)=2^{q+1}\cdot \frac{c-1}{2}.$$
It follows that
$$b=2^{q+1}\cdot \frac{c-1}{2}+2^{q}(r+1)-(2^{q}-1)r=a_{q, \frac{c-1}{2}}\in \{a_{i, k_{1}}: k_{1}\in \mathbb{Z}, i\in\mathbb{Z}^{+}\}.$$
Hence
\begin{equation}\label{2.4}
\{a_{i, k_{1}}: k_{1}\in \mathbb{Z}, i\in\mathbb{Z}^{+}\}=2\mathbb{Z}+r.
\end{equation}
Therefore there exist integers $j$ and $l$ with $j\geq 1$ such that
$$k=a_{j, l}=2^{j+1}l+2^{j}(r+1)-(2^{j}-1)r.$$
By Lemma \ref{lem2}, we have
\begin{eqnarray*}
&&\chi_{A}(4l+2(r+1)-r)+\chi_{A}(4l+2(r+1)-(r+\frac{m}{2}))\\
=&&2+R_{\{\overline{r}, \overline{r+\frac{m}{2}}\}}(\overline{4l+2(r+1)})-\chi_{A}(2l+(r+1))-\chi_{A}(2l+(r+1)+\frac{m}{2}).
\end{eqnarray*}
It means that
\begin{eqnarray}\label{2.5}
&&\chi_{A}(a_{1, l})+\chi_{A}(a_{1, l}+\frac{m}{2}) \nonumber \\
=&&2+R_{\{\overline{r}, \overline{r+\frac{m}{2}}\}}(\overline{4l+2(r+1)})-\chi_{A}(2l+(r+1))-\chi_{A}(2l+(r+1)+\frac{m}{2}).
\end{eqnarray}
By choosing $t=2l+(r+1)$ in (\ref{2.2}), we have
$$\chi_{A}(2l+(r+1))+\chi_{A}(2l+(r+1)+\frac{m}{2})=1.$$
Thus we can write (\ref{2.5}) as
\begin{equation}\label{2.6}
\chi_{A}(a_{1, l})+\chi_{A}(a_{1, l}+\frac{m}{2})=1+R_{\{\overline{r}, \overline{r+\frac{m}{2}}\}}(\overline{a_{1, l}+r}).
\end{equation}
By Lemma \ref{lem2}, (\ref{2.3}) and (\ref{2.6}), we have
\begin{eqnarray*}
&&\chi_{A}(k)+\chi_{A}(k+\frac{m}{2}) \\
=&&\chi_{A}(a_{j, l})+\chi_{A}(a_{j, l}+\frac{m}{2})\\
=&&\chi_{A}(2a_{j-1, l}-r)+\chi_{A}(2a_{j-1, l}-(r+\frac{m}{2}))\\
=&&2+R_{\{\overline{r}, \overline{r+\frac{m}{2}}\}}(\overline{2a_{j-1, l}})-\chi_{A}(a_{j-1, l})-\chi_{A}(a_{j-1, l}+\frac{m}{2})\\
=&&\cdots\\
=&&\sum\limits_{i=1}^{j-1}(-1)^{j-1-i}2+\sum\limits_{i=1}^{j-1}(-1)^{j-1-i}R_{\{\overline{r}, \overline{r+\frac{m}{2}}\}}(\overline{2a_{i, l}})+(-1)^{j-1}(\chi_{A}(a_{1, l})+\chi_{A}(a_{1, l}+\frac{m}{2}))  \\
=&&\sum\limits_{i=1}^{j-1}(-1)^{j-1-i}2+\sum\limits_{i=1}^{j-1}(-1)^{j-1-i}R_{\{\overline{r}, \overline{r+\frac{m}{2}}\}}(\overline{2a_{i, l}})+(-1)^{j-1}(1+R_{\{\overline{r}, \overline{r+\frac{m}{2}}\}}(\overline{a_{1, l}+r})).
\end{eqnarray*}
If $R_{\{\overline{r}, \overline{r+\frac{m}{2}}\}}(\overline{2a_{i, l}})\geq 1$ for some integer $1\leq i\leq j-1$, then
$\overline{2a_{i, l}}=\overline{2r}$ or $\overline{2a_{i, l}}=\overline{2r+\frac{m}{2}}$. Thus $\overline{k}=\overline{a_{j, l}}=\overline{r}$ or $\overline{k}=\overline{a_{j, l}}=\overline{r+\frac{m}{2}}$, a contradiction. Hence $R_{\{\overline{r}, \overline{r+\frac{m}{2}}\}}(\overline{2a_{i, l}})=0$ for any integer $1\leq i\leq j-1$. Therefore
$$\sum\limits_{i=1}^{j-1}(-1)^{j-1-i}R_{\{\overline{r}, \overline{r+\frac{m}{2}}\}}(\overline{2a_{i, l}})=0.$$
If $R_{\{\overline{r}, \overline{r+\frac{m}{2}}\}}(\overline{a_{1, l}+r})\geq 1$, then $\overline{a_{1, l}}=\overline{r}$ or $\overline{a_{1, l}}=\overline{r+\frac{m}{2}}$. Thus $\overline{k}=\overline{a_{j, l}}=\overline{r}$ or $\overline{k}=\overline{a_{j, l}}=\overline{r+\frac{m}{2}}$, a contradiction. Hence $R_{\{\overline{r}, \overline{r+\frac{m}{2}}\}}(\overline{a_{1, l}+r})=0$. It follows that
$$ \chi_{A}(k)+\chi_{A}(k+\frac{m}{2})=\sum\limits_{i=1}^{j-1}(-1)^{j-1-i}2+(-1)^{j-1}=1.$$

This completes the proof of Lemma \ref{lem3}.
\end{proof}

\section{Proof of Theorem \ref{thm1}}

It is clear that $R_{A}(\overline{n})=R_{B}(\overline{n})$ for all $\overline{n}\in \mathbb{Z}_{m}$ if $B=A+\overline{\frac{m}{2}}$.
Now we suppose that $R_{A}(\overline{n})=R_{B}(\overline{n})$ for all $\overline{n}\in \mathbb{Z}_{m}$. Let $A\cap B=\{\overline{r_{1}}, \overline{r_{2}}\}$ with $\overline{r_{1}}\neq \overline{r_{2}}$. If $m=2$, then $B=A+\overline{\frac{m}{2}}$. Now we assume that $m=2^{\alpha}$ with $\alpha\geq 3$. By Lemma \ref{lem3}, it suffices to prove that $\overline{r_{2}}=\overline{r_{1}}+\overline{\frac{m}{2}}$. We suppose that $\overline{r_{2}}\neq \overline{r_{1}}+\overline{\frac{m}{2}}$ and will show that this leads to a contradiction.

{\bf Case 1.} $2\mid (r_{2}-r_{1})$. For any integer $t$ with $t\equiv r_{1}+1\pmod 2$ and any integer $q$ with $n=t+r_{1}+q(r_{2}-r_{1})$ in Lemma \ref{lem2}, we have
$$\chi_{A}(t+q(r_{2}-r_{1}))+\chi_{A}(t+(q-1)(r_{2}-r_{1}))=1+R_{\{\overline{r_{1}}, \overline{r_{2}}\}}(\overline{t+r_{1}+q(r_{2}-r_{1})}).$$
If $R_{\{\overline{r_{1}}, \overline{r_{2}}\}}(\overline{t+r_{1}+q(r_{2}-r_{1})})\geq 1$, then $\overline{t+r_{1}+q(r_{2}-r_{1})}\in\{\overline{2r_{1}}, \overline{2r_{2}}, \overline{r_{1}+r_{2}}\}$. It implies that $t\equiv r_{1}\pmod 2$, which is impossible. Thus $R_{\{\overline{r_{1}}, \overline{r_{2}}\}}(\overline{t+r_{1}+q(r_{2}-r_{1})})=0$ and
\begin{equation}\label{3.1}
\chi_{A}(t+q(r_{2}-r_{1}))+\chi_{A}(t+(q-1)(r_{2}-r_{1}))=1.
\end{equation}
It follows that for any integer $t$ with $t\equiv r_{1}+1\pmod 2$ and any integer $k$
\begin{eqnarray}\label{3.2}
\chi_{A}(t)+\chi_{A}(t+k(r_{2}-r_{1}))=1,\; && \text{ if } 2\nmid k; \nonumber \\
\chi_{A}(t)=\chi_{A}(t+k(r_{2}-r_{1})),\; && \text{ if } 2\mid k.
\end{eqnarray}
Noting that $\overline{r_{2}}\neq \overline{r_{1}}$ and $\overline{r_{2}}\neq \overline{r_{1}}+\overline{\frac{m}{2}}$, we have $(m, r_{2}-r_{1})\mid {\frac{m}{2}}$. Then there exists an even integer $h$ such that
\begin{equation}\label{3.3}
h(r_{2}-r_{1})\equiv \frac{m}{2} \pmod m.
\end{equation}

If $\frac{r_{1}+r_{2}}{2}\equiv r_{1}+1\pmod 2$, then by choosing $t=\frac{r_{1}+r_{2}}{2}$ in (\ref{3.2}), we have
$$\chi_{A}(\frac{r_{1}+r_{2}}{2})=\chi_{A}(\frac{r_{1}+r_{2}}{2}+h(r_{2}-r_{1}))=\chi_{A}(\frac{r_{1}+r_{2}}{2}+\frac{m}{2}).$$
Let $n=r_{1}+r_{2}$ in Lemma \ref{lem2}, we have $R_{\{\overline{r_{1}}, \overline{r_{2}}\}}(\overline{r_{1}+r_{2}})=1$ and
$$2=\chi_{A}(r_{1})+\chi_{A}(r_{2})=3-\chi_{A}(\frac{r_{1}+r_{2}}{2})-\chi_{A}(\frac{r_{1}+r_{2}}{2}+\frac{m}{2}),$$
which is clearly false.

If $\frac{r_{1}+r_{2}}{2}\not\equiv r_{1}+1\pmod 2$, then $\frac{r_{1}+r_{2}}{2}\equiv r_{1}\pmod 2$. Thus $r_{2}\equiv r_{1}\pmod 4$. It follows that for any integers $t\equiv r_{1}+1\pmod 2$ and $j\in\{0, 1, 2\}$, we can obtain
$$2t+j(r_{2}-r_{1})\not\equiv 2r_{1}\; ({\rm\text{mod }} 4), 2t+j(r_{2}-r_{1})\not\equiv 2r_{2}\; ({\rm\text{mod }} 4), 2t+j(r_{2}-r_{1})\not\equiv r_{1}+r_{2}\; ({\rm\text{mod }} 4).$$
Then $R_{\{\overline{r_{1}}, \overline{r_{2}}\}}(\overline{2t+j(r_{2}-r_{1})})=0$ for $j\in\{0, 1, 2\}$. Let $n=2t+j(r_{2}-r_{1})$ for $j\in\{0, 1, 2\}$ in Lemma \ref{lem2}, we have
\begin{eqnarray*}
&&\chi_{A}(2t-r_{1})+\chi_{A}(2t-r_{2})=2-\chi_{A}(t)-\chi_{A}(t+\frac{m}{2}),\\
&&\chi_{A}(2t+r_{2}-2r_{1})+\chi_{A}(2t-r_{1})=2-\chi_{A}(t+\frac{r_{2}-r_{1}}{2})-\chi_{A}(t+\frac{r_{2}-r_{1}}{2}+\frac{m}{2}),\\
&&\chi_{A}(2t+2r_{2}-3r_{1})+\chi_{A}(2t+r_{2}-2r_{1})=2-\chi_{A}(t+r_{2}-r_{1})-\chi_{A}(t+r_{2}-r_{1}+\frac{m}{2}).
\end{eqnarray*}
By (\ref{3.2}) and (\ref{3.3}), we have
\begin{eqnarray*}
&&\chi_{A}(t)=\chi_{A}(t+h(r_{2}-r_{1}))=\chi_{A}(t+\frac{m}{2}),\\
&&\chi_{A}(t+\frac{r_{2}-r_{1}}{2})=\chi_{A}(t+\frac{r_{2}-r_{1}}{2}+h(r_{2}-r_{1}))=\chi_{A}(t+\frac{r_{2}-r_{1}}{2}+\frac{m}{2}),\\
&&\chi_{A}(t+r_{2}-r_{1})=\chi_{A}(t+r_{2}-r_{1}+h(r_{2}-r_{1}))=\chi_{A}(t+r_{2}-r_{1}+\frac{m}{2}).
\end{eqnarray*}
Then
\begin{eqnarray*}
&&\chi_{A}(2t-r_{1})=\chi_{A}(2t-r_{2})=1-\chi_{A}(t),\\
&&\chi_{A}(2t+r_{2}-2r_{1})=\chi_{A}(2t-r_{1})=1-\chi_{A}(t+\frac{r_{2}-r_{1}}{2}),\\
&&\chi_{A}(2t+2r_{2}-3r_{1})=\chi_{A}(2t+r_{2}-2r_{1})=1-\chi_{A}(t+r_{2}-r_{1}).
\end{eqnarray*}
Thus
$$\chi_{A}(t)=1-\chi_{A}(2t-r_{1})=1-\chi_{A}(2t+r_{2}-2r_{1})=1-\chi_{A}(2t+2r_{2}-3r_{1})=\chi_{A}(t+r_{2}-r_{1}).$$
However, by (\ref{3.2}), we have
$$\chi_{A}(t)+\chi_{A}(t+r_{2}-r_{1})=1,$$
a contradiction.

{\bf Case 2.} $2\nmid (r_{2}-r_{1})$. Without loss of generality, we suppose that $2\mid r_{1}$ and $2\nmid r_{2}$. For any nonnegative integer $k$,
let $n=r_{1}+r_{2}+2k(r_{2}-r_{1}), 2r_{1}+2k(r_{2}-r_{1})$ in Lemma \ref{lem2} respectively, we have
\begin{equation}\label{3.4}
\chi_{A}(r_{1}+(2k+1)(r_{2}-r_{1}))+\chi_{A}(r_{1}+2k(r_{2}-r_{1}))=1+R_{\{\overline{r_{1}}, \overline{r_{2}}\}}(\overline{r_{1}+r_{2}+2k(r_{2}-r_{1})})
\end{equation}
and
\begin{eqnarray*}
&&\chi_{A}(r_{1}+2k(r_{2}-r_{1}))+\chi_{A}(r_{1}+(2k-1)(r_{2}-r_{1}))\\
=&&2+R_{\{\overline{r_{1}}, \overline{r_{2}}\}}(\overline{2r_{1}+2k(r_{2}-r_{1})})-\chi_{A}(r_{1}+k(r_{2}-r_{1}))-\chi_{A}(r_{1}+k(r_{2}-r_{1})+\frac{m}{2}).
\end{eqnarray*}
Noting that $\frac{m}{2}(r_{2}-r_{1})\equiv \frac{m}{2}\pmod m$, we have
\begin{eqnarray}\label{3.5}
&&\chi_{A}(r_{1}+2k(r_{2}-r_{1}))+\chi_{A}(r_{1}+(2k-1)(r_{2}-r_{1}))\nonumber\\
=&&2+R_{\{\overline{r_{1}}, \overline{r_{2}}\}}(\overline{2r_{1}+2k(r_{2}-r_{1})})-\chi_{A}(r_{1}+k(r_{2}-r_{1}))\nonumber\\
&&-\chi_{A}(r_{1}+(k+\frac{m}{2})(r_{2}-r_{1})).
\end{eqnarray}
By choosing $k=1$ in (\ref{3.5}), we have
\begin{equation}\label{3.6}
\chi_{A}(r_{1}+2(r_{2}-r_{1}))+\chi_{A}(r_{1}+(1+\frac{m}{2})(r_{2}-r_{1}))=1.
\end{equation}
For $l\in\{1, 2, \ldots, \frac{m}{4}-1\}$, we have
\begin{eqnarray*}
&&\chi_{A}(r_{1}+(4l+2)(r_{2}-r_{1}))+2l+1\\
&=&\chi_{A}(r_{1}+(4l+2)(r_{2}-r_{1}))+\sum\limits_{k=1}^{2l}(\chi_{A}(r_{1}+(2k+1)(r_{2}-r_{1}))+\chi_{A}(r_{1}+2k(r_{2}-r_{1}))+\chi_{A}(r_{2})\\
&=&\sum\limits_{k=1}^{2l+1}(\chi_{A}(r_{1}+2k(r_{2}-r_{1}))+\chi_{A}(r_{1}+(2k-1)(r_{2}-r_{1})))\\
&=& 4l+3-\sum\limits_{k=1}^{2l+1}\chi_{A}(r_{1}+k(r_{2}-r_{1}))-\sum\limits_{k=1}^{2l+1}\chi_{A}(r_{1}+(k+\frac{m}{2})(r_{2}-r_{1})))\\
&=& 4l+3-\chi_{A}(r_{2})-\sum\limits_{k=1}^{l}(\chi_{A}(r_{1}+(2k+1)(r_{2}-r_{1}))+\chi_{A}(r_{1}+2k(r_{2}-r_{1})))\\
&&\hspace{2mm}-\sum\limits_{k=1}^{l}(\chi_{A}(r_{1}+(2k+1+\frac{m}{2})(r_{2}-r_{1}))+
\chi_{A}(r_{1}+(2k+\frac{m}{2})(r_{2}-r_{1})))\\
&&\hspace{2mm} -\chi_{A}(r_{1}+(1+\frac{m}{2})(r_{2}-r_{1}))\\
&=&2l+2-\chi_{A}(r_{1}+(1+\frac{m}{2})(r_{2}-r_{1})).
\end{eqnarray*}
Then
\begin{equation}\label{3.7}
\chi_{A}(r_{1}+(1+\frac{m}{2})(r_{2}-r_{1}))+\chi_{A}(r_{1}+(4l+2)(r_{2}-r_{1}))=1
\end{equation}
By (\ref{3.6}) and (\ref{3.7}), for $l\in \{0, 1, 2, \ldots, \frac{m}{4}-1\}$, we have
\begin{equation}\label{3.8}
\chi_{A}(r_{1}+(1+\frac{m}{2})(r_{2}-r_{1}))+\chi_{A}(r_{1}+(4l+2)(r_{2}-r_{1}))=1.
\end{equation}

By choosing $k=\frac{m}{4}$ in (\ref{3.4}) and $k=0$ in (\ref{3.5}), we have
$$ \chi_{A}(r_{1}+(1+\frac{m}{2})(r_{2}-r_{1}))+\chi_{A}(r_{1}+\frac{m}{2}(r_{2}-r_{1}))=1,$$
and
$$\chi_{A}(r_{1}+(-1)(r_{2}-r_{1}))+\chi_{A}(r_{1}+\frac{m}{2}(r_{2}-r_{1}))=1.$$
Thus
\begin{equation}\label{3.9}
\chi_{A}(r_{1}+(1+\frac{m}{2})(r_{2}-r_{1}))=\chi_{A}(r_{1}+(-1)(r_{2}-r_{1})).
\end{equation}
For $l\in\{1, 2, \ldots, \frac{m}{4}-1\}$, we have
\begin{eqnarray*}
&&\chi_{A}(r_{1}+4l(r_{2}-r_{1}))+2l\\
&=&\chi_{A}(r_{1}+4l(r_{2}-r_{1}))+\sum\limits_{k=1}^{2l-1}(\chi_{A}(r_{1}+(2k+1)(r_{2}-r_{1}))+\chi_{A}(r_{1}+2k(r_{2}-r_{1})))+\chi_{A}(r_{2})\\
&=&\sum\limits_{k=1}^{2l}(\chi_{A}(r_{1}+2k(r_{2}-r_{1}))+\chi_{A}(r_{1}+(2k-1)(r_{2}-r_{1})))\\
&=& 4l+1-\sum\limits_{k=1}^{2l}\chi_{A}(r_{1}+k(r_{2}-r_{1}))-\sum\limits_{k=1}^{2l}\chi_{A}(r_{1}+(k+\frac{m}{2})(r_{2}-r_{1})))\\
&=& 4l+1-\chi_{A}(r_{2})-\sum\limits_{k=1}^{l-1}(\chi_{A}(r_{1}+(2k+1)(r_{2}-r_{1}))+\chi_{A}(r_{1}+2k(r_{2}-r_{1})))-\chi_{A}(r_{1}+2l(r_{2}-r_{1}))\\
&&\hspace{2mm}-\sum\limits_{k=1}^{l-1}(\chi_{A}(r_{1}+(2k+1+\frac{m}{2})(r_{2}-r_{1}))+
\chi_{A}(r_{1}+(2k+\frac{m}{2})(r_{2}-r_{1})))\\
&&\hspace{2mm} -\chi_{A}(r_{1}+(1+\frac{m}{2})(r_{2}-r_{1}))-\chi_{A}(r_{1}+(2l+\frac{m}{2})(r_{2}-r_{1}))\\
&=&2l+2-\chi_{A}(r_{1}+2l(r_{2}-r_{1}))-\chi_{A}(r_{1}+(1+\frac{m}{2})(r_{2}-r_{1}))-\chi_{A}(r_{1}+(2l+\frac{m}{2})(r_{2}-r_{1})).
\end{eqnarray*}
Then
$$ \chi_{A}(r_{1}+4l(r_{2}-r_{1}))+\chi_{A}(r_{1}+2l(r_{2}-r_{1}))+\chi_{A}(r_{1}+(1+\frac{m}{2})(r_{2}-r_{1}))+\chi_{A}(r_{1}+(2l+\frac{m}{2})(r_{2}-r_{1}))=2.
$$
By choosing $k=2l$ in (\ref{3.5}), we have
$$ \chi_{A}(r_{1}+4l(r_{2}-r_{1}))+\chi_{A}(r_{1}+(4l-1)(r_{2}-r_{1}))+\chi_{A}(r_{1}+2l(r_{2}-r_{1}))+\chi_{A}(r_{1}+(2l+\frac{m}{2})(r_{2}-r_{1}))=2.$$
Thus
\begin{equation}\label{3.10}
\chi_{A}(r_{1}+(1+\frac{m}{2})(r_{2}-r_{1}))=\chi_{A}(r_{1}+(4l-1)(r_{2}-r_{1})).
\end{equation}
By (\ref{3.9}) and (\ref{3.10}), for $l\in \{0, 1, 2, \ldots, \frac{m}{4}-1\}$, we have
\begin{equation}\label{3.11}
\chi_{A}(r_{1}+(1+\frac{m}{2})(r_{2}-r_{1}))=\chi_{A}(r_{1}+(4l-1)(r_{2}-r_{1})).
\end{equation}

By choosing $k=4l+2$ for $l\in \{0, 1, 2, \ldots, \frac{m}{8}-1\}$ in (\ref{3.5}), we have
\begin{eqnarray}\label{3.12}
&&\chi_{A}(r_{1}+(8l+4)(r_{2}-r_{1}))+\chi_{A}(r_{1}+(8l+3)(r_{2}-r_{1}))\nonumber\\
&=&2-\chi_{A}(r_{1}+(4l+2)(r_{2}-r_{1}))-\chi_{A}(r_{1}+(4l+2+\frac{m}{2})(r_{2}-r_{1})).
\end{eqnarray}
By (\ref{3.8}), (\ref{3.11}) and (\ref{3.12}), we have
\begin{equation}\label{3.13}
\chi_{A}(r_{1}+(8l+4)(r_{2}-r_{1}))=\chi_{A}(r_{1}+(1+\frac{m}{2})(r_{2}-r_{1})).
\end{equation}

If $\alpha=3$, then by choosing $l=0$ in (\ref{3.13}) and $k=2$ in (\ref{3.4}), we have
$$\chi_{A}(r_{1}+4(r_{2}-r_{1}))=\chi_{A}(r_{1}+5(r_{2}-r_{1}))$$
and
$$\chi_{A}(r_{1}+5(r_{2}-r_{1}))+\chi_{A}(r_{1}+4(r_{2}-r_{1}))=1,$$
which is impossible. If $\alpha\geq 4$, then by choosing $k=4$ in (\ref{3.5}), we have
$$ \chi_{A}(r_{1}+8(r_{2}-r_{1}))+\chi_{A}(r_{1}+7(r_{2}-r_{1}))=2-\chi_{A}(r_{1}+4(r_{2}-r_{1}))-\chi_{A}(r_{1}+(4+\frac{m}{2})(r_{2}-r_{1})).$$
By (\ref{3.11}) and (\ref{3.13}), we have
$$ \chi_{A}(r_{1}+8(r_{2}-r_{1}))=2-3\chi_{A}(r_{1}+(1+\frac{m}{2})(r_{2}-r_{1})),$$
which is false.

This completes the proof of Theorem \ref{thm1}.

\end{document}